\documentclass[a4paper,reqno]{amsart}
\usepackage{amssymb}
\usepackage{graphics} 
\usepackage{graphicx}
\usepackage[colorlinks=true]{hyperref}

\hypersetup{urlcolor=blue, citecolor=red}

\allowdisplaybreaks
\numberwithin{equation}{section}

\newtheorem{theorem}{Theorem}[section]

\newtheorem{lemma}[theorem]{Lemma}

\newtheorem{definition}[theorem]{Definition}
\newtheorem{remark}[theorem]{Remark}

\newcommand{\epsi}{\varepsilon}
\newcommand{\Gr}{G}

\newcommand{\D}{\ensuremath{\mathcal{D}}}

\newcommand{\G}{\ensuremath{\mathcal{G}}}
\newcommand{\F}{\ensuremath{\mathcal{F}}}

\newcommand{\Real}{\mathbb{R}}

\DeclareMathOperator{\id}{Id}

\newcommand{\dott}{\,\cdot\,}

\newcommand{\sign}{\mathop{\rm sign}}

\newcommand{\norm}[1]{\left\Vert#1\right\Vert}

\newcommand{\abs}[1]{\left\vert#1\right\vert}


\email{katrin.grunert@univie.ac.at}
\email{holden@math.ntnu.no}
\email{xavierra@cma.uio.no}

\begin{document}

\title[Lipschitz metric for the Camassa--Holm
system]{Lipschitz metric for the two-component Camassa--Holm
  system}

\maketitle

\centerline{\scshape Katrin Grunert} \medskip {\footnotesize
  \centerline{Department of Mathematical Sciences}
  \centerline{Norwegian University of Science and
    Technology} \centerline{NO-7491 Trondheim, Norway}
} 

\medskip
\centerline{\scshape Helge Holden }
\medskip 

{\footnotesize \centerline{Department of Mathematical Sciences} \centerline{Norwegian
    University of Science and Technology} \centerline{NO-7491 Trondheim, Norway} \centerline{and}
  \centerline{Centre of Mathematics for Applications} \centerline{ University of Oslo}
  \centerline{NO-0316 Oslo, Norway} }

\medskip
\centerline{\scshape Xavier Raynaud } 
\medskip 

{\footnotesize \centerline{Centre of Mathematics for Applications} \centerline{ University of Oslo}
  \centerline{NO-0316 Oslo, Norway} }

\begin{abstract}
  We construct a Lipschitz metric for conservative solutions of the Cauchy problem on the line for
  the two-component Camassa--Holm system $u_t-u_{txx}+3uu_x-2u_xu_{xx}-uu_{xxx}+\rho\rho_x=0$, and
  $\rho_t+(u\rho)_x=0$ with given initial data $(u_0, \rho_0)$.  The Lipschitz metric $d_{\D^M}$ has the
  property that for two solutions $z(t)=(u(t),\rho(t),\mu_t)$ and $\tilde z(t)=(\tilde u(t),\tilde
  \rho(t),\tilde \mu_t)$ of the system we have $d_{\D^M}(z(t),\tilde z(t))\le C_{M,T}
  d_{\D^M}(z_0,\tilde z_0)$ for $t\in[0,T]$. Here the measure $\mu_t$ is such that its absolutely
  continuous part equals the energy $(u^2+u_x^2+\rho^2)(t)dx$, and the solutions are restricted to a
  ball of radius $M$.
\end{abstract}

\section{Introduction}\label{sec:intro}

The two-component Camassa--Holm (2CH) system, which was first derived in
\cite[Eq.~(43)]{OlverRosenau}, is given by
\begin{subequations}
  \label{eq:chsys}
  \begin{align}
    \label{eq:chsys11}
    u_t-u_{txx}+3uu_x-2u_xu_{xx}-uu_{xxx}+\rho\rho_x&=0,\\
    \label{eq:chsys12}
    \rho_t+(u\rho)_x&=0,
  \end{align}
\end{subequations}
or equivalently
\begin{subequations}
  \label{eq:rewchsys10}
  \begin{align}
    \label{eq:rewchsys11}
    u_t+uu_x+P_x&=0,\\
    \label{eq:rewchsys12}
    \rho_t+(u\rho)_x&=0,
  \end{align}
\end{subequations}
where $P$ is implicitly defined by
\begin{equation}
  \label{eq:rewchsys13}
  P-P_{xx}=u^2+\frac12u_x^2+\frac12\rho^2.
\end{equation}
The Camassa--Holm equation \cite{CH,CHH} is obtained by considering the case when $\rho$ vanishes
identically.  The aim of this article is to present the construction of a Lipschitz metric for this
system on the real line with vanishing asymptotics, that is, $u\in H^1$ and $\rho\in L^2$. The
conservative solutions to \eqref{eq:rewchsys10} are constructed in \cite{GHR:12} for nonvanishing
asymptotics. A Lipschitz metric for the system with periodic boundary conditions is given in
\cite{GHR:12b}. We here combine the two approaches by constructing a Lipschitz metric for
conservative, decaying solutions. The preservation of the energy is needed in the proofs so that the
constuction of the metric only applies to vanishing asymptotics. Here we rather describe and
motivate the general ideas behind the construction, which we hope can be of interest in the study of
other related equations. For more background on the two-component Camassa--Holm system, we refer to
\cite{GHR:12} and the references therein. For related papers, see
\cite{BC,BHR,HolRay:06a,HolRay:06b}.

\section{Relaxation of the equations by the introduction of Lagrangian coordinates}
\label{sec:semilag}

The change of coordinates from Eulerian to Lagrangian coordinates has relaxation properties which
are well-known for the Burgers equation, viz.
\begin{equation}
  \label{eq:burger}
  u_t+uu_x=0.
\end{equation}
Lagrangian coordinates are defined by characteristics
\begin{equation*}
  y_t(t,\xi)=u(t,y(t,\xi)),
\end{equation*}
which give the position of a particle which moves in the velocity field $u$ and its velocity, known
as the Lagrangian velocity, is given by
\begin{equation*}
  U(t,\xi)=u(t,x), \quad x=y(t,\xi).
\end{equation*}
The method of characteristics consists of rewriting \eqref{eq:burger} in terms of the Lagrangian
variables and yields
\begin{equation}
  \begin{aligned}
    y_t&=U,\\
    U_t&=0.
  \end{aligned}\label{eq:linburger}
\end{equation}
Comparing \eqref{eq:burger} to \eqref{eq:linburger}, we observe that we start with a
\textit{nonlinear} and \textit{partial} (derivatives with respect to $t$ and $x$) differential
equation and end up with a \textit{linear} and \textit{ordinary} (derivative only with respect to
$t$) differential equation. We get rid of the nonlinear convection term, and \eqref{eq:linburger} is
nothing but Newton's law, which states that the acceleration is constant in the absence of forces. A
well-known drawback of the change of coordinates from Eulerian to Lagrangian coordinates is that it
doubles the dimension of the problem: We start with a scalar equation and end up with a system of
dimension two. This is an important issue and we will deal with it in Section \ref{sec:relab}.
However, in return, we gain the possibility to represent a larger class of objects or, more
precisely in our case, to increase the regularity of the unknown functions. Let us make this
imprecise statement clearer by an example and, to do so, we drop the dependence in $t$ in the
notation, as we look at singularities in the space variable. The function $u(x)$ can be represented
by its graph $(x,u(x))$ but this graph can itself be represented as a parametric curve, namely,
$(y(\xi), U(\xi))$ and, as we know, the set of graphs is smaller than the set of parametric
curves. As far as regularity is concerned, the Heaviside function
\begin{equation*}
  h(x)=
  \begin{cases}
    0&\text{ if }x<0,\\
    1&\text{ if }x\geq 0,
  \end{cases}
\end{equation*}
is only of bounded variation but it can be represented in Lagrangian coordinates by the following
pair of more regular (in this case Lipschitz) functions
\begin{align}
  \label{eq:initcollpeakon}
  y(\xi)&=
  \begin{cases}
    \xi&\text{ if }\xi<0,\\
    0&\text{ if }\xi\in[0,1),\\
    \xi - 1&\text{ if }\xi\geq1,\\
  \end{cases}
  &H(\xi)&=
  \begin{cases}
    0&\text{ if }\xi<0,\\
    \xi&\text{ if }\xi\in[0,1),\\
    1&\text{ if }\xi\geq1
  \end{cases}
\end{align}
Indeed, $(x,h(x))$ and $(y(\xi), H(\xi))$ represent one and the same curve, except for the vertical line
joining the origin to the point $(0,1)$. We will return to this example later. The solution of the
Camassa--Holm equation (i.e., where $\rho$ vanishes identically) experiences in general wave
breaking (i.e., loss of of regularity in the sense that the spatial derivative becomes unbounded
while keeping the $H^1$ norm finite) in finite time (\cite{cons:98,cons:98b,cons:00}) and the
antisymmetric peakon-antipeakon solution, which is described in \cite{HolRay:06b} and depicted in
Figure \ref{fig:coll}, helps us to understand how the solutions can be prolonged in a way which
preserves the energy.
\begin{figure}
  \begin{center}
    \includegraphics[width=6cm]{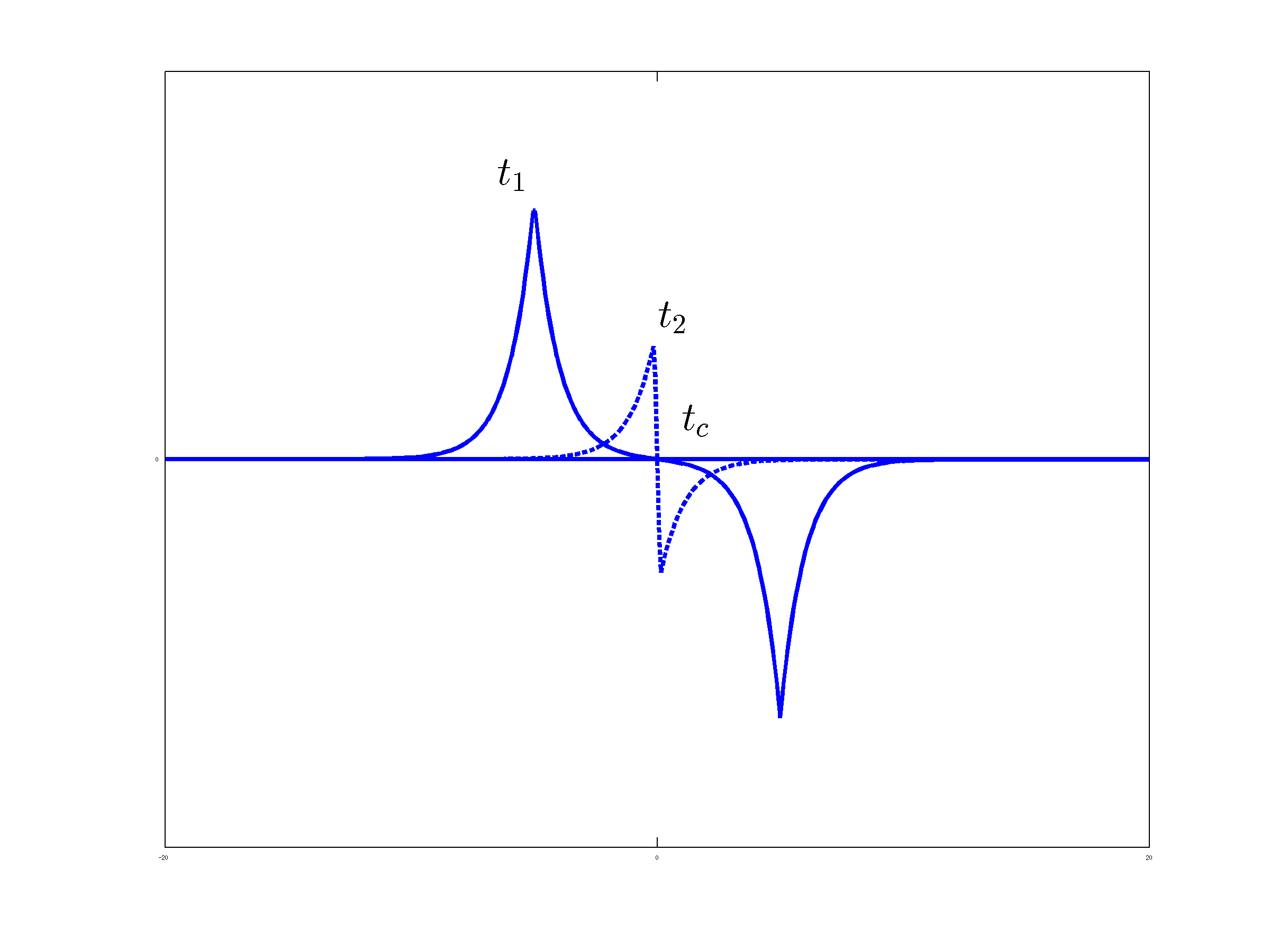}
    \includegraphics[width=6cm]{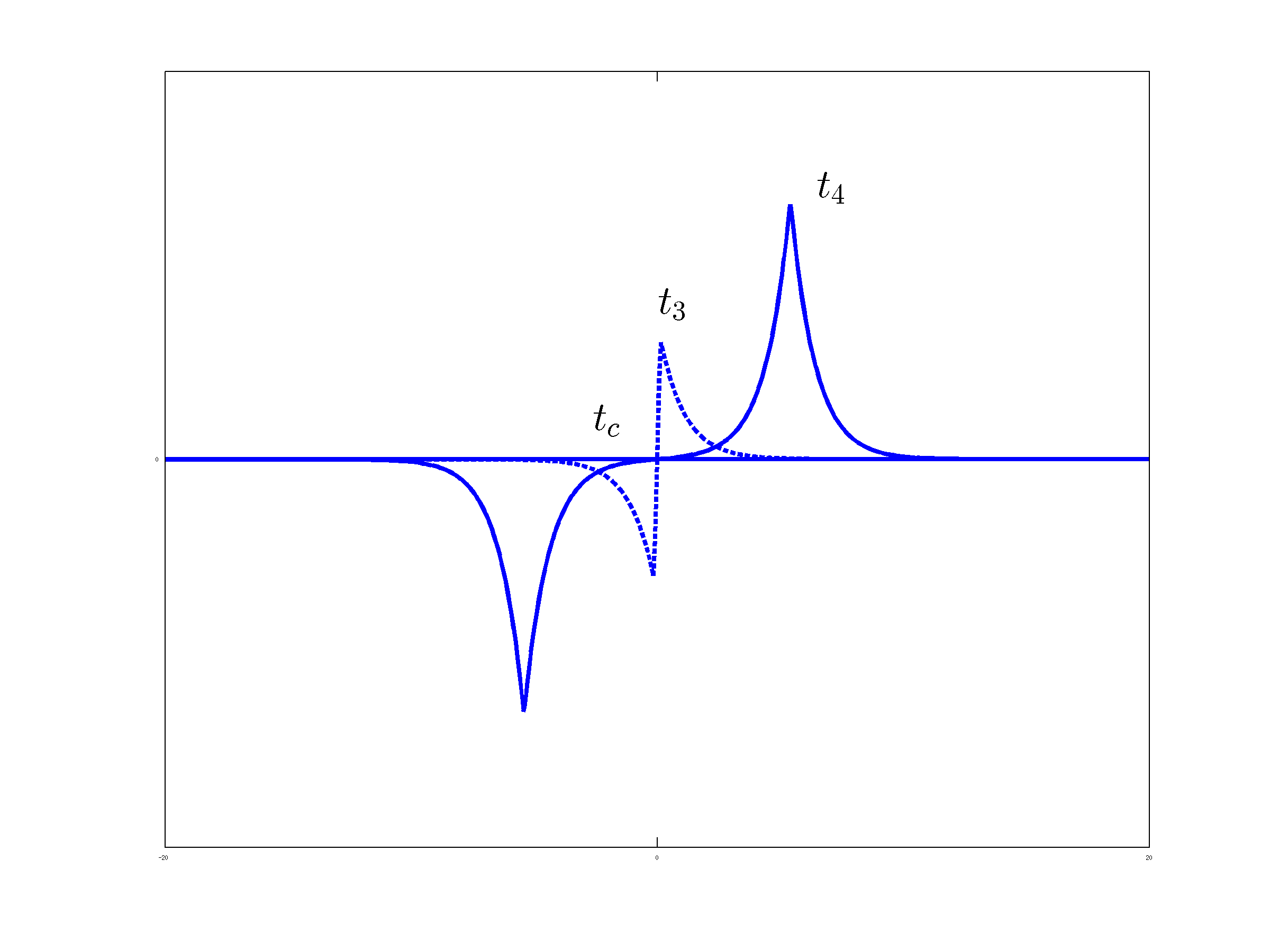}
    \caption{Anti symmetric peakon-antipeakon collision, before (on the left) and after (on the
      right) collision.}
    \label{fig:coll}
  \end{center}
\end{figure}

At collision time $t_c$, we have
\begin{align*}
  \lim_{t\to t_c}u(t, x)&=0\text{ in }L^\infty,&\lim_{t\to t_c} u_x(t,0) = -\infty,
\end{align*}
while the $H^1$ norm is constant so that $\lim_{t\to
  t_c}\norm{u(t,\dott)}_{H^1}=\norm{u(0,\dott)}_{H^1}$. To obtain the conservative solution, we need
to track the amount and the location of the concentrated energy. The function $u$ alone cannot
provide this information as $u(t_c,\dott)$ is identically zero. Thus, we have to introduce an extra
variable to describe the solutions. In Lagrangian variables, it takes the form of the
\textit{cumulative energy} $H(t,\xi)$, which is given by
\begin{equation}
  H(t,\xi)=\int_{-\infty}^{y(t,\xi)} (u^2+u_x^2+\rho^2)(x) dx.
\end{equation} 
We will introduce later its counter-part in Eulerian variables. Equation \eqref{eq:chsys12}
transports the density $\rho$. Formally, after changing variables, we have $\rho(x)\,dx =
\rho(y)\,dy = \rho(y)y_\xi\,d\xi$, so that the Lagrangian variable corresponding to $\rho$ is given
by
\begin{equation}\label{eq:rhotor}
  r(t,\xi)=\rho(t,y(t,\xi))y_\xi(t,\xi).
\end{equation}

Next, we rewrite \eqref{eq:rewchsys10} in the Lagrangian variables $(y,U,H,r)$. We obtain the
following system
\begin{equation}
  \label{eq:equivsys}
  \begin{aligned} \zeta_t &= U,\\ U_t &= -Q,\\ H_t &= U^3-2PU,\\r_t&=0,
  \end{aligned}
\end{equation} 
where $\zeta (t,\xi)=y(t,\xi)-\xi$,
\begin{equation}\label{repP} P(t,\xi)=\frac{1}{4} \int_{\Real} \exp(-\vert y(t,\xi)-y(t,\eta)\vert )
  (U^2y_\xi +H_\xi)(t,\eta)d\eta,
\end{equation} 
and
\begin{equation}\label{repQ} Q(t,\xi)=-\frac{1}{4} \int_\Real \sign(y(t,\xi)-y(t,\eta)) \exp(-\vert
  y(t,\xi)-y(t,\eta)\vert ) (U^2y_\xi +H_\xi)(t,\eta)d\eta.
\end{equation} 
See \cite{GHR:12} for more details on this derivation. After differentiation, we obtain
\begin{subequations}
  \label{eq:govsysder}
  \begin{align}
    \label{eq:govsysder1}
    y_{\xi t}& =U_\xi,\\
    \label{eq:govsysder3}
    U_{\xi t}&=\frac{1}{2}H_\xi+(\frac12 U^2-P)y_\xi, \\
    \label{eq:govsysder4}
    H_{\xi t}&=(3U^2-2P)U_\xi - 2QUy_\xi,\\
    \label{eq:govsysder5}
    r_t&=0.
  \end{align}
\end{subequations}
This system is semilinear and we recognize some features observed earlier for the Burgers equation:
We start from a nonlinear partial differential equation and we end up with a system of ordinary
differential equations which is \textit{semilinear}. We consider the system as an ordinary
differential equation because the order of the spatial derivative is the same on both sides of the
equation, so that the existence and uniqueness of solutions can be established by a contraction
argument. Finally, it is important to recall in this section the geometric nature of the
Camassa--Holm equation. The equation is a geodesic in the group of diffeomorphism for the $H^1$
norm, see, e.g., \cite{cons:01b}, as the Burgers equation for the $L^2$ norm. Using the connection
between geometry and fluid mechanics, as presented in \cite{arnold}, the function $t\mapsto
y(t,\xi)$ can then be understood as a path in the group of diffeomorphisms. Thus besides the
relaxation properties we have just described, this interpretation adds a direct geometrical
relevance to use of Lagrangian coordinates, see also \cite{EKL:11} for the system.

\section{Semigroup in Lagrangian coordinates}\label{sec:lag}

In \cite[Theorem 3.2]{GHR:12}, we prove by a contraction argument that short-time solutions to
\eqref{eq:equivsys} exist in a Banach space, which we will here denote $E$ and define as
follows. Let $V$ be the Banach space defined by
\begin{equation*} V=\{f\in L^\infty \ |\ f_\xi\in L^2\}
\end{equation*} 
and the norm of $V$ is given by $\norm{f}_V=\norm{f}_{L^\infty}+\norm{f_\xi}_{L^2}$. We set $E$
\begin{equation*} E=V\times H^1\times V\times L^2
\end{equation*} with the following norm
$\norm{X}=\norm{\zeta}_{V}+\norm{U}_{H^1}+\norm{H}_{V}+\norm{r}_{L^2}$ for any $X=(\zeta, U, H, r)\in E$. Given a constant $M>0$, we denote by $B_M$ the ball
\begin{equation}
  \label{eq:defBM} B_M=\{X\in E\ |\ \norm{X}\leq M\}.
\end{equation}

Short-time solutions of \eqref{eq:govsysder} cannot in general be extended to global solutions. The
challenge is to identify an appropriate set of initial data for which one can construct global
solutions that at the same time preserve the structure of the equations, allowing us to return to
the Eulerian variables. There are intrinsic relations between the variables in \eqref{eq:govsysder}
that need to be conserved by the solution. This is handled by the set $\G$ defined below.  In
particular, the set $\G$ is preserved by the flow.

\begin{definition}
  \label{def:F} The set $\G$ is composed of all $(\zeta, U, H, r)\in E$ such that
  \begin{subequations}
    \label{eq:lagcoord}
    \begin{align}
      \label{eq:lagcoord1} &(\zeta,U,H,r)\in \left[W^{1,\infty}\right]^3\times L^\infty,\\
      \label{eq:lagcoord2} &y_\xi\geq0, H_\xi\geq0, y_\xi+H_\xi>0 \text{ almost everywhere, and
      }\lim_{\xi\rightarrow-\infty}H(\xi)=0,\\
      \label{eq:lagcoord3} &y_\xi H_\xi=y_\xi^2U^2+U_\xi^2+r^2\text{ almost everywhere},
    \end{align}
  \end{subequations} where we denote $y(\xi)=\zeta(\xi)+\xi$.
\end{definition}
The condition $y_\xi\geq 0$ implies that the mapping $\xi\mapsto y(\xi)$ is \textit{almost} a
diffeomorphism. The solution develop singularities exactly when this mapping ceases to be a
diffeomorphism, that is, when $y_\xi=0$ in some regions. The condition \eqref{eq:lagcoord3} shows
that the variables $(y,U,H,r)$ are strongly coupled. In fact, when $y_\xi\neq0$, we can recover $H$
from \eqref{eq:lagcoord3}. It reflects the fact that $H_\xi$ represents, in Lagrangian coordinates,
the energy density of $u$ and $\rho$ (that is, $(u^2+u_x^2+\rho^2)dx$ in Eulerian coordinates) and
therefore, when the solution is smooth, it can be computed from the variables $y$, $U$, and
$r$. Note that the coupling between $H$ and $(y,U,r)$ disappears when $y_\xi=0$, which is precisely
the moment when collisions occur and when we need the information $H$ provides on the energy to
prolong the solution. The identity makes also clear the smoothing property of the Camassa--Holm
system. If $r_0\geq c>0$ for some constant $c$, this property is preserved and then $y_\xi$ never
vanishes. The solution keeps the same degree of regularity it has initially, see \cite{GHR:12}.

As in \cite[Theorem 3.6]{GHR:12}, we obtain the Lipschitz continuity of the semigroup
\begin{theorem}
  \label{th:global} For any $\bar X=(\bar y,\bar U,\bar H, \bar r)\in\G$, the system
  \eqref{eq:equivsys} admits a unique global solution $X(t)=(y(t),U(t),H(t), r(t))$ in
  $C^1(\Real_+,E)$ with initial data $\bar X=(\bar y,\bar U,\bar H,\bar r)$. We have $X(t)\in\G$ for
  all times. If we equip $\G$ with the topology induced by the $E$-norm, then the mapping
  $S\colon\G\times\Real_+\to\G$ defined by
  \begin{equation*} S_t(\bar X)=X(t)
  \end{equation*} is a Lipschitz continuous semigroup. More precisely, given $M>0$ and $T>0$, there exists a
  constant $C_M$ which depends only on $M$ and $T$ such that, for any two elements
  $X_\alpha,X_\beta\in\G\cap B_M$, we have
  \begin{equation}
    \label{eq:stabSt} \norm{S_tX_\alpha-S_tX_\beta}\leq C_M\norm{X_\alpha-X_\beta}
  \end{equation} for any $t\in[0,T]$.
\end{theorem}

\section{Relabeling symmetry}\label{sec:relab}

The equations are well-posed in Lagrangian coordinates. We want to transport this result back to
Eulerian coordinates. If the two sets of coordinates were in bijection, then it would be
straightforward but, as mentioned earlier, Lagrangian coordinates increase the number of unknowns
from two ($u$ and $\rho$) to four (the components of $X$), which indicates that such a bijection
does not exist. There exists a redundancy in Lagrangian coordinates and the goal of this section is
precisely to identify this redundancy, in order to be able to define the correct equivalence
classes. This redundancy is also present in the case of the Burgers equation when we define the
Cauchy problem for both \eqref{eq:burger} and \eqref{eq:linburger}. To the initial condition
$u(0,x)=u_0(x)$ for \eqref{eq:burger}, there corresponds infinitely many parametrizations of the
initial conditions for \eqref{eq:linburger} given by
\begin{align*}
  y(0,\xi)&=f(\xi),&U(0,\xi)=u_0(f(\xi)),
\end{align*}
for an arbitrary diffeomorphism $f$. As also mentioned earlier, the representation of a graph is
uniquely defined by a single function while there are infinitely many different parametrizations of
any given curve. We will use the term \textit{relabeling} for this lack of uniqueness in the
characterization of one and the same curve.

We now define the relabeling functions as follows.
\begin{definition}
  \label{def:Gr} We denote by $\Gr$ the subgroup of the group of homeomorphisms from $\Real$ to
  $\Real$ such that
  \begin{subequations}
    \label{eq:Hcond}
    \begin{align}
      \label{eq:Hcond1} f-\id\text{ and }f^{-1}-\id &\text{ both belong to }W^{1,\infty},\\
      \label{eq:Hcond2} f_\xi-1& \text{ belongs to }L^2,
    \end{align}
  \end{subequations} where $\id$ denotes the identity function. Given $\kappa>0$, we denote by
  $\Gr_\kappa$ the subset $\Gr_\kappa$ of $\Gr$ defined by
  \begin{equation*} \Gr_\kappa=\{f\in\Gr\ |\
    \norm{f-\id}_{W^{1,\infty}}+\norm{f^{-1}-\id}_{W^{1,\infty}}\leq\kappa\}.
  \end{equation*}
\end{definition}
We refine the definition of $\G$ in Definition \ref{def:F} by introducing the subsets $\F_\kappa$
and $\F$ as
\begin{equation*}
  \F_\kappa=\{X=(y,U,H,r)\in\G\ |\ y+H\in \Gr_\kappa\},
\end{equation*} and
\begin{equation}
  \label{eq:defF}
  \F=\{X=(y,U,H,r)\in\G\ |\ y+H\in \Gr\}.
\end{equation} 
The regularity requirement on the relabeling functions given in Definition \ref{def:Gr} and the
definition of $\F$ are introduced in order to be able to define the action of $\Gr$ on $\F$, that is,
for any $X=(y,U,H,r)\in\F$ and any function $f\in\Gr$, the function $(y\circ f,U\circ f,H\circ f,
r\circ f f_\xi)$ belongs to $\F$ and we will denote it by $X\circ f$.  This corresponds to the
relabeling action. Note that relabeling acts differently on \textit{primary} functions, as $y$, $U$
and $H$ (in this case, we have $(U, f)\mapsto U\circ f$) and on \textit{derivatives} or
\textit{densities}, as $y_\xi$, $U_\xi$, $H_\xi$ and $r$ (in that case we have $(r, f)\mapsto r\circ
f f_\xi$). The space $\F$ is preserved by the governing equation \eqref{eq:equivsys} and, as
expected,
the semigroup of solutions in Lagrangian coordinates preserves relabeling, i.e., we have the
following result.
\begin{lemma}[{\cite[Theorem 4.8]{GHR:12}}]
  \label{lem:equivPi} The mapping $S_t$ is equivariant, that is,
  \begin{equation*} S_t(X\circ f)=S_t(X)\circ f
  \end{equation*}
  for any $X\in\F$ and $f\in \Gr$.
\end{lemma}

Now that we have identified the redundancy of Lagrangian coordinates as the action of relabeling, we
want to handle it by considering equivalence classes. However, equivalence classes are rather
abstract objects which will be hard to work with from an analytical point of view. We consider
instead the section defined by $\F_0$, which contains one and only one representative for each
equivalence class, so that the quotient $\F/\Gr$ is in bijection with $\F_0$. Let us denote by $\Pi$
the projection of $\F$ into $\F_0$ defined as
\begin{equation*}
  \Pi(X)=X\circ(y+H)^{-1}
\end{equation*} 
for any $X=(y,U,H,r)\in\F$. By definition, we have that $X$ and $\Pi(X)$ belong to the same
equivalence class. We can check that the mapping $\Pi$ is a projection, i.e., $\Pi\circ\Pi = \Pi$,
and that it is also invariant, i.e., $\Pi(X\circ f)=\Pi(X)$. It follows that the mapping $[X]\mapsto
\Pi(X)$ is a bijection from $\F/\Gr$ to $\F_0$. 

\section{Eulerian coordinates}\label{sec:euler}

In the method of characteristics, once the equation is solved in Lagrangian coordinates, we recover
the solution in Eulerian coordinates by setting $u(t,x) = U(t,y^{-1}(t,x))$, where $y^{-1}(t,x)$
denotes---assuming it exists---the inverse of $\xi\mapsto y(t,\xi)$. The Burgers equation and the
Camassa--Holm equation develop singularity because $y$ does not remain invertible. In the case of
the Burgers equation, $u$ becomes discontinuous but the Camassa--Holm equation enjoys more
regularity and $u$ remains continuous. This is a consequence of the preservation of the $H^1$ norm,
but it can also be seen from the Lagrangian point of view. Indeed, even if $y$ is not invertible, we
can define $u(t,x)$ as
\begin{equation*}
  u(t,x) = U(t,\xi)\text{ for any }\xi\text{ such that }x=y(t,\xi).
\end{equation*}
This is well-defined because if there exist $\xi_1$ and $\xi_2$ such that $x=y(t,\xi_1)=y(t,\xi_2)$,
then $y_\xi(t,\xi)=0$ for all $\xi\in[\xi_1,\xi_2]$ because $y$ is non-decreasing, see
\eqref{eq:lagcoord2}. Then, by \eqref{eq:lagcoord3}, we get $U_\xi(t,\xi)=0$ so that
$U(t,\xi_1)=U(t,\xi_2)$. Furthermore, as we explained earlier in the case of a peakon-antipeakon
collision, some information is needed about the energy to prolong the solution after 
collision. If $y$ is invertible, we recover the energy density in Eulerian coordinates as
\begin{equation}
  \label{eq:euldefenerg1}
  (u^2+u_x^2+\rho^2)\,dx= \frac{H_\xi}{y_\xi}\circ y^{-1}\,d\xi,
\end{equation}
which corresponds to the push-forward of the measure $H_\xi\,d\xi$ with respect to $y$, i.e.,
\begin{equation}
  \label{eq:euldefenerg2}
  (u^2+u_x^2+\rho^2)\,dx = y_\#(H_\xi\,d\xi).
\end{equation}
However, when $y$ is not invertible \eqref{eq:euldefenerg1} cannot be used and $y_\#(H_\xi\,d\xi)$
may not be absolutely continuous so that \eqref{eq:euldefenerg2} will not hold either. It motivates
the introduction of the energy $\mu$ defined here as $y_\#(H_\xi\,d\xi)$, which represents the
energy of the system. The set $\D$ of Eulerian coordinates is defined as follows.
\begin{definition}
  \label{def:D} The set $\D$ consists of all triples $(u,\rho,\mu)$ such that
  \begin{enumerate}
  \item $u\in H^1$, $\rho\in L^2$, and
  \item $\mu$ is a positive Radon measure whose absolutely continuous part, $\mu_{ac}$, satisfies
    \begin{equation}
      \label{eq:relmumuac}
      \mu_{ac}=(u^2+u_x^2+\rho^2)dx.
    \end{equation}
  \end{enumerate}
\end{definition}

It can be shown (see \cite[Section 4]{GHR:12}) that the identity \eqref{eq:lagcoord3} is somehow
equivalent to \eqref{eq:relmumuac} but it is clear that, from an analytical point of view, it easier
to deal with an algebraic identity like \eqref{eq:lagcoord3} than with a property like
\eqref{eq:relmumuac} which immediately requires tools from measure theory. We can show that $\D$ and
$\F_0$ are in bijection, and the mappings between the two are given in the following definition. The
first one has been already explained.
\begin{definition} Given any element $X $ in $\F_0$, then $(u,\rho,\mu)$ defined as follows
  \begin{equation*}
    u(x)=U(\xi) \text{ for any } \xi \text{ such that } x=y(\xi),
  \end{equation*}
  \begin{equation*}
    \rho(x)=y_\#(r d\xi),\quad \mu=y_\#(H_\xi d\xi),
  \end{equation*} belongs to $\D$. We denote by $M:\F_0\to \D$ the map which to any $X$ in $\F_0$
  associates $(u,\rho,\mu)$.
\end{definition}
The mapping, which we denoted by $L$, from $\D$ to $\F_0$ is defined as follows.
\begin{definition}
  \label{def:L}
  For any $(u,\rho,\mu)$ in $\D$ let
  \begin{equation}
    \label{eq:defL}
    \left\{
      \begin{aligned} y(\xi)&=\sup\{y \mid \mu((-\infty,y))+y<\xi\},\\ 
        H(\xi)& =\xi-y(\xi),\\
        U(\xi)&=u\circ y(\xi),\\
        r(\xi)&=\rho\circ y(\xi)y_\xi(\xi).
      \end{aligned} \right.
  \end{equation}
\end{definition}
We can see that the lack of regularity of $u$, which will occur when $\mu$ is singular or very
large, is transformed into regions where the function $y$ is constant or almost constant. Using the
relabeling degree of freedom, we manage to rewrite functions in $L^2$ and measures as bounded
functions (in $L^\infty$). For example, for the peakon-antipeakon collision depicted in Figure
\ref{fig:coll}, the initial data given by $u_0(x)=\rho_0(x)=0$ and $\mu=\delta(x)\,dx$, which
corresponds to the collision time, $t_c$, when the total energy is equal to one, yields
$r(\xi)=U(\xi)=0$ with $y(\xi)$ and $H(\xi)$ as defined in \eqref{eq:initcollpeakon}. We can check that,
in this case $\delta(x)\,dx=y_\#(H_\xi\,d\xi)$. Finally, we define the semigroup $T_t$ of
conservative solutions in the original Eulerian variables $\D$ as
\begin{equation*}
  T_t:=M\Pi S_tL. 
\end{equation*}

\section{Lipschitz metric for the semigroup}\label{sec:lip}

We apply the construction of the semigroup $T_t$ in Section \ref{sec:euler}, and we can check, as
done in \cite[Theorem 5.2]{GHR:12}, that, for given initial data $(u_0,\rho_0,\mu_0)$, if we denote
$(u(t), \rho(t), \mu_t)=T_t(u_0,\rho_0,\mu_0)$, then $(u, \rho)$ are weak solutions to
\eqref{eq:rewchsys10}.  Moreover,
\begin{equation*}
  \mu_t(\Real)=\mu_0(\Real)
\end{equation*}
so that the solutions are conservative. Our goal is to define a metric on $\D$ which makes the
semigroup Lipschitz continuous. The Lipschitz continuity is a property of a semigroup which can be
used to establish its uniqueness, see \cite{Bressan} and \cite[Theorem 2.9]{Bressan:book}. By our construction, a metric for the semigroup $T_t$ is
readily available. We can simply transport the topology of the Banach space $E$ from $\F_0$ to $\D$
and obtain, for two elements $(u,\rho,\mu)$ and $(\tilde u,\tilde \rho,\tilde \mu)$,
\begin{equation}
  \label{eq:defdD1}
  d_\D\big((u,\rho,\mu),(\tilde u,\tilde \rho,\tilde \mu)\big) = \norm{L(u,\rho,\mu)-L(\tilde u,\tilde \rho,\tilde \mu)}_E.
\end{equation}
We have
\begin{equation*}
  d_\D\big(T_t(u,\rho,\mu),T_t(\tilde u,\tilde \rho,\tilde \mu)\big) = \norm{\Pi S_tL(u,\rho,\mu)-\Pi S_t L(\tilde u,\tilde \rho,\tilde \mu)}_E.
\end{equation*}
It can be proven that the projection $\Pi$ is continuous (see \cite[Lemma 4.6]{GHR:12}), but it is
not Lipschitz (at least, we have been unable to prove it). Thus, even if $S_t$ is Lipschitz
continuous, the semigroup $T_t$ is only continuous with respect to the metric $d_\D$ defined by
\eqref{eq:defdD1}. In the definition \eqref{eq:defdD1} of the metric, we let the section $\F_0$ play
a special role, but this section is arbitrarily chosen. The set $\F_0$ is by construction nonlinear
(because of \eqref{eq:lagcoord3}) and to use a linear norm to measure distances does not respect
that. In fact, we want to measure the distance between equivalence classes. A natural starting point
is to define, for $X_\alpha,X_\beta\in\F$, $\bar J(X_\alpha,X_\beta)$ as
\begin{equation}
  \label{eq:defJXX}
  \bar J(X_\alpha,X_\beta)=\inf_{f, g\in\Gr}\norm{X_\alpha\circ f-X_\beta\circ
    g}.
\end{equation}
The function $\bar J$ is relabeling invariant, that is, $\bar J(X_\alpha\circ f,X_\beta\circ g)=\bar
J(X_\alpha,X_\beta)$ and  measures precisely the distance between two equivalence classes. However, we
have to deal with the fact that the linear norm of $E$ does not play well with relabeling: It is not
invariant with respect to relabeling, i.e., we do not have
\begin{equation}
  \label{eq:invnorm}
  \norm{X\circ f} = \norm{X}.
\end{equation}
However,  such a norm exists. Let
\begin{equation*}
  B=\{X\in L^\infty\ |\ X_\xi \in L^1\}.
\end{equation*}
Then,
\begin{equation*}
  \norm{X\circ f}_{B} =\norm{X\circ f}_{L^\infty} + \norm{X_\xi\circ f f_\xi}_{L^1} = \norm{X}_{L^\infty} + \norm{X_\xi}_{L^1} = \norm{X}_{B}.
\end{equation*}
To cope with the lack of relabeling invariance of $\bar J$, we introduce $J$ defined as follows.
\begin{definition}
  \label{def:J} 
  Let $X_\alpha,X_\beta\in\F$, we define $J(X_\alpha,X_\beta)$ as
  \begin{equation}
    \label{eq:defJ} J(X_\alpha,X_\beta)=\inf_{f_1,f_2\in\Gr}\big(\norm{X_\alpha\circ
      f_1-X_\beta}+\norm{X_\alpha-X_\beta\circ f_2}\big).
  \end{equation}
\end{definition}
The function $J$ is not relabeling invariant, but we have $J(X_\alpha,X_\beta)=0$ if $X_\alpha$ and
$X_\beta$ both belong to the same equivalence class. Moreover, the relabeling invariance is not strictly
needed for our purpose and the following weaker property is enough. Given $X_\alpha,X_\beta\in \F$
and $f\in\Gr_\kappa$, we have
\begin{equation}
  \label{eq:Jrelabbound} J(X_\alpha\circ f,X_\beta\circ f)\leq CJ(X_\alpha,X_\beta)
\end{equation} 
for some constant $C$ which depends only on $\kappa$, see \cite{GHR:13}.  Note that, if the norm $E$
were invariant, that is, \eqref{eq:invnorm} were fulfilled, then the function $J$ and $\bar J$ would be
equivalent, because we would have $\bar J\leq J\leq 2\bar J$.

\begin{remark}
  We will make use of the following notation. The variable $X$ is used as a standard notation for
  $(y,U,H,r)$. By the $L^\infty$ norm of $X$, we mean
  \begin{equation}
    \label{eq:notation1}
    \norm{X}_{L^\infty}=\norm{y-\id}_{L^\infty}+\norm{U}_{L^\infty}+\norm{H}_{L^\infty},
  \end{equation}
  and, by the $L^2$ norm of the derivative $X_\xi$, we mean
  \begin{equation}
    \label{eq:notation2}
    \norm{X_\xi}_{L^2}=\norm{y_\xi-1}_{L^2}+\norm{U_\xi}_{L^2}+\norm{H_\xi}_{L^2}+\norm{r}_{L^2},
  \end{equation}
  and, similarly,
  \begin{equation}
    \label{eq:notation3}
    \norm{X_\xi}_{L^\infty}=\norm{y_\xi-1}_{L^\infty}+\norm{U_\xi}_{L^\infty}+\norm{H_\xi}_{L^\infty}+\norm{r}_{L^\infty}.
  \end{equation}
\end{remark}
From $J$, we obtain a metric $d$ by the following construction.
\begin{definition} Let $X_\alpha,X_\beta\in\F_0$, we define $d(X_\alpha,X_\beta)$ as
  \begin{equation}
    \label{eq:defdist} d(X_\alpha,X_\beta)=\inf \sum_{i=1}^NJ(X_{n-1},X_n)
  \end{equation} where the infimum is taken over all finite sequences $\{X_n\}_{n=0}^N\subset\F_0$ which
  satisfy $X_0=X_\alpha$ and $X_N=X_\beta$.
\end{definition}

\begin{lemma}
  \label{lem:distance}
  The mapping $d:\F_0\times\F_0\to\Real_+$ is a distance on $\F_0$, which is bounded as follows
  \begin{equation}
    \label{eq:dequiv} \frac12\norm{X_\alpha-X_\beta}_{L^\infty}\leq
    d(X_\alpha,X_\beta)\leq2\norm{X_\alpha-X_\beta}.
  \end{equation}
\end{lemma}

\begin{proof} 
  The first part of the proof is identical to \cite{GHR:13} and we reproduce it here for
  convenience. For any $X_\alpha,X_\beta\in\F_0$, we have
  \begin{equation}
    \label{eq:linfcompj} \norm{X_\alpha-X_\beta}_{L^\infty}\leq 2 J(X_\alpha,X_\beta).
  \end{equation} We have
  \begin{align} \notag \norm{X_\alpha-X_\beta}_{L^\infty}&\leq\norm{X_\alpha-X_\alpha\circ
      f}_{L^\infty}+\norm{X_\alpha\circ f-X_\beta}_{L^{\infty}}\\
    \label{eq:xamxbediffg} &\leq
    \norm{X_{\alpha,\xi}}_{L^\infty}\norm{f-\id}_{L^\infty}+\norm{X_\alpha\circ
      f-X_\beta}_{L^{\infty}}.
  \end{align} 
  It follows from the definition of $\F_0$ that $0\leq y_\xi\leq1$, $0\leq H_\xi\leq1$
  and $\abs{U_\xi}\leq1$ so that $\norm{X_{\alpha,\xi}}_{L^\infty}\leq 3$. We also have
  \begin{equation}
    \label{eq:bdfminid}
    \norm{f-\id}_{L^\infty}=\norm{(y_\alpha+H_\alpha)\circ f-(y_\beta+H_\beta)}_{L^\infty}\leq\norm{X_\alpha\circ f-X_\beta}_{L^\infty}.
  \end{equation} Hence, from \eqref{eq:xamxbediffg}, we get
  \begin{equation}
    \label{eq:bdlinfour}
    \norm{X_\alpha-X_\beta}_{L^\infty}\leq4\norm{X_\alpha\circ
      f-X_\beta}_{L^\infty}.
  \end{equation} In the same way, we obtain
  $\norm{X_\alpha-X_\beta}_{L^\infty}\leq4\norm{X_\alpha-X_\beta\circ f}_{L^\infty}$ for any
  $f\in\Gr$. After adding these two last inequalities and taking the infimum, we get
  \eqref{eq:linfcompj}. For any $\epsi>0$, we consider a finite sequence $\{X_n\}_{n=0}^N\subset\F_0$ such that
  $X_0=X_\alpha$ and $X_N=X_\beta$ and $\sum_{i=1}^NJ(X_{n-1},X_n)\leq d(X_\alpha,X_\beta)+\epsi$. We
  have
  \begin{align*} \norm{X_\alpha-X_\beta}_{L^\infty}&\leq
    \sum_{n=1}^{N}\norm{X_{n-1}-X_n}_{L^\infty}\\ &\leq2\sum_{n=1}^{N}J(X_{n-1},X_n)\\
    &\leq2(d(X_\alpha,X_\beta)+\epsi).
  \end{align*} After letting $\epsi$ tend to zero, we get
  \begin{equation}
    \label{eq:LinfbdJ} 
    \norm{X_\alpha-X_\beta}_{L^\infty}\leq 2 d(X_\alpha,X_\beta).
  \end{equation}
  The second inequality in \eqref{eq:dequiv} follows from the definitions of $J$ and $d$. Indeed, we
  have
  \begin{equation*}
    d(X_\alpha,X_\beta)\leq J(X_\alpha,X_\beta)\leq2\norm{X_\alpha-X_\beta}.
  \end{equation*}
  It is left to prove that $d$ defines a metric. The symmetry is intrinsic in the definition of $J$
  while the construction of $d$ from $J$ takes care of the triangle inequality. From
  \eqref{eq:dequiv}, we get that $d(X_\alpha,X_\beta)=0$ implies $(y_\alpha, U_\alpha,
  H_\alpha)=(y_\beta, U_\beta, H_\beta)$. By \eqref{eq:lagcoord3}, we get that
  $r_\alpha^2=r_\beta^2$, but we cannot yet conclude that $r_\alpha=r_\beta$. Let us define
  $R_\alpha(\xi)=\int_{-\infty}^\xi r_\alpha(\eta)e^{-\abs{\eta}}\,d\eta$ and
  $R_\beta(\xi)=\int_{-\infty}^\xi r_\beta(\eta)e^{-\abs{\eta}}\,d\eta$. Then, we have, for any
  $f\in\Gr$,
  \begin{multline}
    R_\alpha(\xi) - R_\beta(\xi) = -\int_{\xi}^{f(\xi)}r_\alpha(\eta) e^{-\abs{\eta}}\,d\eta +
    \int_{-\infty}^\xi r_\alpha\circ f f_\xi (e^{-\abs{f(\eta)}}-e^{-\abs{\eta}})\,d\eta\\ +
    \int_{-\infty}^\xi(r_\alpha\circ f f_\xi -r_\beta) e^{-\abs{\eta}}\,d\eta,
  \end{multline}
  which implies
  \begin{align*}
    \norm{R_{\alpha}-R_{\beta}}_{L^\infty}& \leq \norm{f-\id}_{L^\infty}+
    \norm{\int_{-\infty}^\xi r_\alpha\circ f f_\xi
      (e^{-\abs{f(\eta)}}-e^{-\abs{\eta}})\,d\eta}_{L^\infty}\\
    &\quad+\norm{r_\alpha\circ f f_\xi -r_\beta}_{L^2}.
  \end{align*}
  We have that
  \begin{align*}
    \int_{-\infty}^\xi r_\alpha\circ f f_\xi (e^{-\abs{f(\eta)}}-e^{-\abs{\eta}})\,d\eta &=
    \int_{-\infty}^\xi r_\alpha\circ f f_\xi
    e^{-\abs{f(\eta)}}(1-e^{\abs{f(\eta)}-\abs{\eta}})\,d\eta
  \end{align*}
  implies
  \begin{align*}
    \norm{\int_{-\infty}^\xi r_\alpha\circ f f_\xi (e^{-\abs{f(\eta)}}-e^{-\abs{\eta}})\,d\eta}_{L^\infty}&\leq \norm{e^{\abs{f(\xi)}-\abs{\xi}} - 1}_{L^\infty}\norm{r_\alpha}_{L^2}\norm{e^{-\abs{\xi}}}_{L^2}\\
    &\leq C\norm{r_\alpha}_{L^2} \norm{f-\id}_{L^\infty},
  \end{align*}
  for $C=e$ if we assume that $\norm{f-\id}_{L^\infty}\leq 1$. Since $X_\alpha\in \F_0$ so that
  $y_\xi\leq 1$, we get from \eqref{eq:lagcoord3} that $\norm{r_\alpha}_{L^2}\leq
  \norm{H_\alpha}_{L^\infty}^{1/2}$. Collecting the results obtained so far, we find that
  \begin{equation}
    \label{eq:bdralphabeta}
    \norm{R_{\alpha}-R_{\beta}}_{L^\infty} \leq (2 + C \norm{H_\alpha}_{L^\infty}^{1/2})\norm{X_\alpha\circ f - X_\beta}
  \end{equation}
  for any $\norm{f-\id}_{L^\infty}\leq 1$. Let us now assume that $d(X_\alpha, X_\beta) = 0$. For
  any $\epsi>0$, we can find a sequence such that
  \begin{equation*}
    \sum_{n=1}^N\norm{X_n\circ f_n - X_{n-1}}\leq \epsi.
  \end{equation*}
  Using \eqref{eq:bdfminid} and \eqref{eq:bdlinfour}, we get $\norm{f_n-\id}_{L^\infty}\leq\epsi$
  and prove by induction that
  \begin{equation}
    \label{eq:indstat}
    \norm{H_n}_{L^\infty}\leq\sum_{i=1}^{n}\norm{X_i\circ f_i - X_{i-1}}_{L^\infty} + \norm{H_\alpha}_{L^{\infty}},
  \end{equation}
  for all $n\leq N$. Indeed, we have
  \begin{align*}
    \norm{H_{n+1}}_{L^\infty} & = \norm{H_{n+1}\circ f_{n+1}}_{L^\infty}\\
    &\leq \norm{H_{n+1}\circ f_{n+1} - H_n}_{L^\infty} + \norm{H_n}_{L^\infty}\\
    &\leq \sum_{i=1}^{n + 1}\norm{X_i\circ f_i - X_{i-1}}_{L^\infty} + \norm{H_\alpha}_{L^{\infty}},
  \end{align*}
  after using the induction hypothesis. From \eqref{eq:indstat}, we get
  \begin{equation*}
    \norm{H_n}_{L^\infty}\leq \epsi + \norm{H_\alpha}.
  \end{equation*}
  Hence, by choosing $\epsi\leq 1$, and using repeatedly \eqref{eq:bdralphabeta}, we obtain
  \begin{align*}
    \norm{R_{\alpha}-R_{\beta}}_{L^\infty} &\leq \sum_{n=1}^N \norm{R_{n} - R_{n-1}}_{L^\infty}\\
    &\leq (2 + C(\epsi + \norm{H_\alpha}_{L^\infty})^{1/2})\sum_{n=1}^N\norm{X_\alpha\circ f - X_\beta}    \\
    &\leq (2 + C(\epsi + \norm{H_\alpha}_{L^\infty})^{1/2})\epsi.
  \end{align*}
  After letting $\epsi$ tend to zero, this last inequality implies that $R_\alpha=R_\beta$ so that
  $r_\alpha=r_\beta$, which concludes the proof that $d$ is a metric.
\end{proof}

The Lipschitz estimate for the semigroup $S_t$ given in \eqref{eq:stabSt} is valid for initial data
in $B_M$. Hence, as we want to use the same Lipschitz estimate for any of the $X_n$ in the sequence
defining the metric in \eqref{eq:defdist}, we have to redefine this metric and require that all
$X_n$ belong to $\F_0\cap B_M$. The problem is  that $B_M$ is not preserved by the semigroup
$S_t$, and we will not be able to use the same distance at later times. This is why we introduce the
set
\begin{equation*}
  \F^M=\{X=(y,U,H,r)\in \F\ |\ \norm{H}_{L^\infty}\leq M\},
\end{equation*}
which is preserved by \textit{both} relabeling and the semigroup. Note that $\F^M$ has a simple
physical interpretation as it corresponds to the set of all solutions which have total energy
bounded by $M$. Moreover, following closely the proof of \cite[Lemma 3.4]{GHR:13}, we obtain that
for $X\in \F_0$, the sets $B_M$ and $\F^M$ are in fact equivalent, i.e., there exists $\bar M$
depending only on $M$ such that
\begin{equation}
  \label{eq:eqFBM}
  \F_0\cap\F^M\subset B_{\bar M}. 
\end{equation}
We set $\F_0^M=\F_0\cap\F^M$ and define the metric $d^M$ as follows.
\begin{definition}\label{def:metric} Let $d^M$ be the distance on $\F_0^M$ which is defined, for any
  $X_\alpha,X_\beta\in\F_0^M$, as
  \begin{equation}
    \label{eq:defdM}
    d^M(X_\alpha,X_\beta)=\inf \sum_{n=1}^NJ(X_{n-1},X_n)
  \end{equation} 
  where the infimum is taken over all finite sequences $\{X_n\}_{n=0}^N\subset\F_0^M$ such that
  $X_0=X_\alpha$ and $X_N=X_\beta$.
\end{definition}
\begin{figure}[h]
  \centering
  \includegraphics[width=10cm]{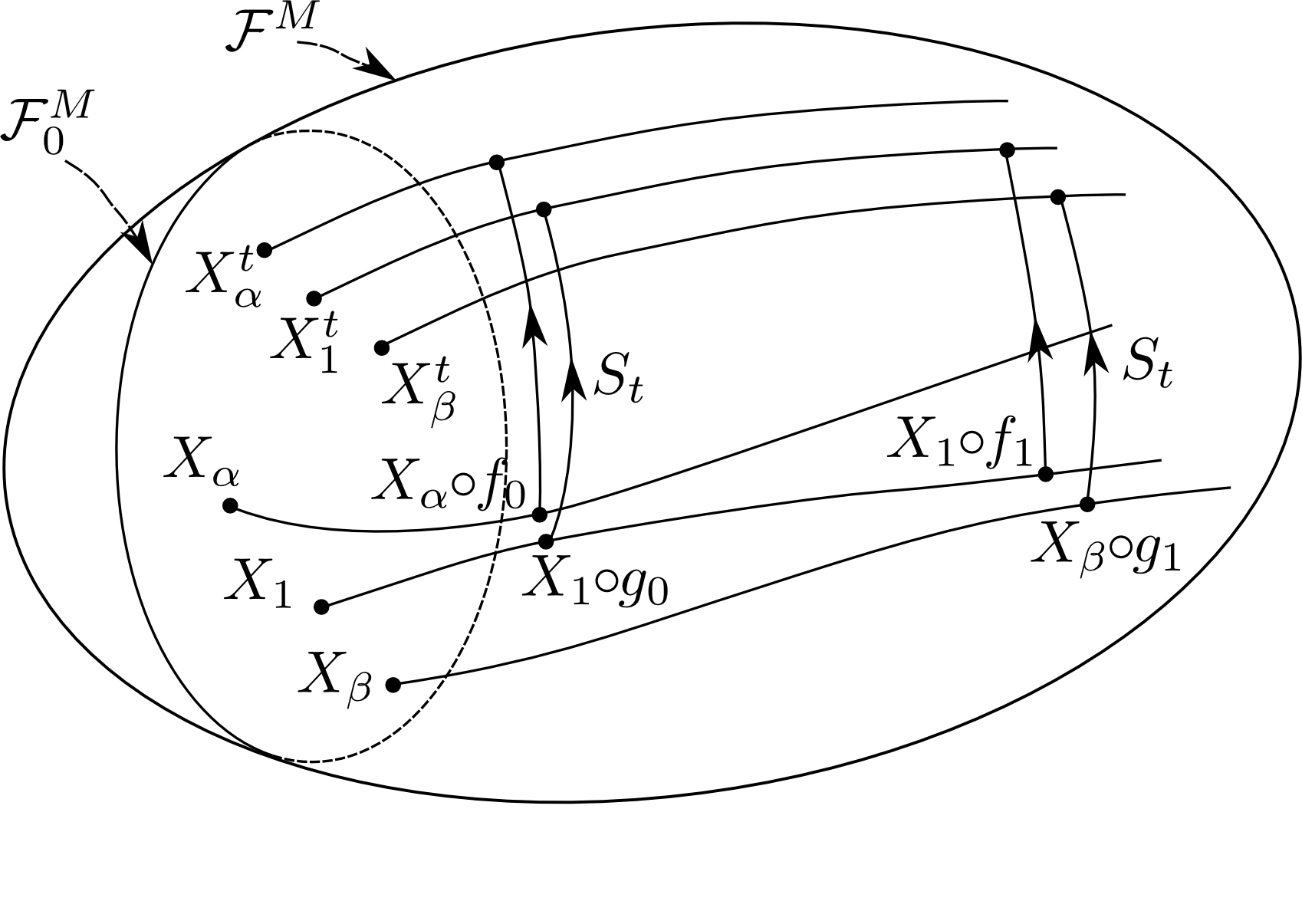}
  \caption{Illustration for the construction of the metric. The \textit{horizontal} curves
    represent points which belong to the same equivalence class.}
  \label{fig:distance}
\end{figure}
We can now state our main stability theorem
\begin{theorem}
  \label{th:stab} Given $T>0$ and $M>0$, there exists a constant $C_{M,T}$ which depends only on $M$
  and $T$ such that, for any $X_\alpha,X_\beta\in\F_0^M$ and $t\in[0,T]$, we have
  \begin{equation}
    \label{eq:stab} d^M(\Pi S_tX_\alpha, \Pi S_tX_\beta)\leq C_{M,T} d^M(X_\alpha,X_\beta).
  \end{equation}
\end{theorem}
In fact due to the use of equivalent notations, the proof of the theorem is identical to
\cite[Theorem 3.6]{GHR:13}. Here, we propose to present a simplified proof where we assume that the
norm of $E$ is invariant with respect to relabeling, that is, \eqref{eq:invnorm} holds. By doing so,
we hope that some general ideas behind the construction of the metric becomes clearer.  Much of the
construction can be understood from the illustration in Figure~\ref{fig:distance}.  In this figure,
we denote $X_\alpha^t=\Pi S_t(X_\alpha\circ f_0)$, $X_\beta^t=\Pi S_t(X_\beta\circ g_1)$ and
$X_1^t=\Pi S_t(X_1\circ g_0)=\Pi S_t(X_1\circ f_1)$. Let us imagine the (very improbable) case where
the infimum in \eqref{eq:defdM} and the infimum in \eqref{eq:defJ} both are reached, so that
$d^M(X_\alpha, X_\beta) = \norm{X_\alpha\circ f_0-X_1\circ g_0} + \norm{X_1\circ f_1-X_\beta\circ
  g_1}$. Then, we have
\begin{align*}
  d^M(X_\alpha^t,X_\beta^t)&\leq J(X_\alpha^t,X_1^t) + J(X_1^t,X_\beta^t)\\
  &= J(S_t(X_\alpha\circ f_0), S_t(X_1\circ g_0)) + J(S_t(X_1\circ f_1),S_t(X_\beta\circ g_1))\\
  &\leq \norm{S_t(X_\alpha\circ f_0)-S_t(X_1\circ g_0)} + \norm{S_t(X_1\circ f_1)-S_t(X_\beta\circ
    g_1)}\\
  &\leq C_{M,T} \big(\norm{X_\alpha\circ f_0-X_1\circ g_0} + \norm{X_1\circ f_1-X_\beta\circ g_1}\big)\\
  &= C_{M,T} d^M(X_\alpha, X_\beta),
\end{align*}
which corresponds to the Lipschitz estimate of Theorem \ref{th:stab}.

\begin{proof}[Simplified proof of Theorem \ref{th:stab}] As we mentioned earlier, when the norm is
  invariant, then $J$ and $\bar J$ are equivalent. Here, it is simpler to consider $\bar J$. For any
  $\epsi>0$, there exist a finite sequence $\{X_n\}_{n=0}^N$ in $\F_0^M$ and functions
  $\{f_n\}_{n=0}^{N-1}$, $\{g_n\}_{n=0}^{N-1}$ in $\Gr$ such that $X_0=X_\alpha$, $X_N=X_\beta$ and
  \begin{equation}
    \label{eq:sumXnm1}
    \sum_{i=1}^N\norm{X_{n-1}\circ f_{n-1}-X_{n}\circ g_{n-1}}\leq
    d_M(X_\alpha,X_\beta)+\epsi.
  \end{equation}
  Since $B_{\bar M}$, where $\bar M$ is defined so that \eqref{eq:eqFBM} holds, is preserved by
  relabeling, we have that $X_{n}\circ f_n$ and $X_{n}\circ g_{n-1}$ belong to $B_{\bar M}$. From
  the Lipschitz stability result given in \eqref{eq:stabSt}, we obtain that
  \begin{equation}
    \label{eq:normSxnm1}
    \norm{S_t(X_{n-1}\circ f_{n-1})-S_t(X_{n}\circ g_{n-1})}\leq C_{M,T} \norm{X_{n-1}\circ f_{n-1}-X_{n}\circ g_{n-1}},
  \end{equation}
  where the constant $C_{M,T}$ depends only on $M$ and $T$. Introduce
  \begin{equation*}
    \bar X_n=X_n\circ f_n,\  \bar X_n^t=S_t(\bar X_n), \text{ for }n=0,\ldots,N-1,
  \end{equation*}
  and
  \begin{equation*}
    \tilde X_n=X_{n}\circ g_{n-1},\ \tilde X_n^t=S_t(\tilde X_n), \text{ for }n=1,\ldots,N.
  \end{equation*}
  Then \eqref{eq:sumXnm1} rewrites as
  \begin{equation}
    \label{eq:sumXnm1b}
    \sum_{i=1}^N\norm{\bar X_{n-1}-\tilde X_{n}}\leq
    d_M(X_\alpha,X_\beta)+\epsi
  \end{equation}
  while \eqref{eq:normSxnm1} rewrites as
  \begin{equation}
    \label{eq:normSnnot}
    \norm{\bar X_{n-1}^t-\tilde X_{n}^t}\leq C_{M,T} \norm{\bar X_{n-1}-\tilde X_{n}}.
  \end{equation}
  We have
  \begin{equation*}
    \Pi(\bar X_0^t)=\Pi\circ S_t(X_0\circ f_0)=\Pi\circ (S_t(X_0)\circ f_0)=\Pi\circ S_t(X_0)=\bar S_t(X_\alpha)
  \end{equation*}
  and similarly $\Pi(\tilde X_N^t)=\Pi S_t(X_\beta)$. We consider the sequence which consists of
  $\{\Pi \bar X_n^t\}_{n=0}^{N-1}$ and $\bar S_t(X_\beta)$. Using the property that $\F^M$ is
  preserved both by relabeling and by the semigroup, we obtain that $\{\Pi \bar X_n^t\}_{n=0}^{N-1}$
  and $\bar S_t(X_\beta)$ belong to $\F^M$ and therefore also to $\F_0^M$. The endpoints are $\Pi
  S_t(X_\alpha)$ and $\Pi S_t(X_\beta)$. From the definition of the metric $d_M$, we get
  \begin{align}
    \notag d_M(\bar S_t(X_\alpha),\bar S_t(X_\beta))&\leq\sum_{n=1}^{N-1} \bar J(\Pi \bar
    X_{n-1}^t,\Pi\bar X_n^t)+\bar J(\Pi \bar
    X_{N-1}^t,\bar S_t(X_\beta))\\
    \label{eq:dtilS}
    &=\sum_{n=1}^{N-1} \bar J(\bar X_{n-1}^t,\bar X_n^t)+\bar J(\bar X_{N-1}^t,\tilde
    X_N^t),
  \end{align}
  due to the invariance of $\bar J$ with respect to relabeling.  By using the
  equivariance of $S_t$, we obtain that
  \begin{equation}
    \label{eq:tilXrelbarX}
    \begin{aligned}
      \tilde X_n^t=S_t(\tilde X_n)&=S_t((\bar X_n\circ f_n^{-1})\circ g_{n-1})\\
      &=S_t(\bar X_n)\circ (f_n^{-1}\circ g_{n-1})=\bar X_n^t\circ (f_n^{-1}\circ g_{n-1}).
    \end{aligned}
  \end{equation}
  Hence we get from \eqref{eq:dtilS} that
  \begin{align*}
    d_M(\bar S_t(X_\alpha),\bar S_t(X_\beta))&\leq\sum_{n=1}^{N-1} \bar J(\bar X_{n-1}^t,\tilde
    X_n^t)+\bar J(\bar
    X_{N-1}^t,\tilde X_N^t)\\
    &\leq \sum_{n=1}^N\norm{\bar X_{n-1}^t-\tilde
      X_n^t} &\text{ by \eqref{eq:dequiv} }\\
    &\leq C_{M,T} \sum_{n=1}^N\norm{\bar X_{n-1}-\tilde
      X_n} &\text{  by \eqref{eq:normSnnot}}\\
    &\leq C_{M,T} (d_M(X_\alpha,X_\beta)+\epsi).
  \end{align*}
  After letting $\epsi$ tend to zero, we obtain \eqref{eq:stab}.
\end{proof}

The Lipschitz stability of the semigroup $T_t$ follows then naturally from Theorem~\ref{th:stab}. It
holds on sets of bounded energy. Let $\D^M$ be the subsets of $\D$ defined as
\begin{equation}
  \D^M=\{ (u,\rho,\mu)\in\D \mid \mu(\Real)\leq M\}.
\end{equation} 
On the set $\D^M$ we define the metric $d_{\D^M}$ as
\begin{equation}\label{eq:defDM}
  d_{\D^M}((u,\rho,\mu),(\tilde u, \tilde\rho, \tilde\mu))
  =d^M  (L(u,\rho,\mu), L(\tilde  u,\tilde\rho,\tilde \mu)),
\end{equation} 
where the metric $d^M$ is defined as in Definition~\ref{def:metric}. This definition is well-posed
as, from the definition of $L$, we have that if $(u,\rho,\mu)\in \D^M$, then $L(u,\rho,\mu)\in
\F_0^M$.
\begin{theorem}
  The semigroup $(T_t, d_\D)$ is a continuous semigroup on $\D$ with respect to the metric
  $d_D$. The semigroup is Lipschitz continuous on sets of bounded energy, that is: Given $M>0$ and a
  time interval $[0,T]$, there exists a constant $C_{M,T}$, which only depends on $M$ and $T$ such
  that for any $(u,\rho,\mu)$ and $(\tilde u,\tilde\rho,\tilde \mu)$ in $\D^M$, we have
  \begin{equation}
    d_{\D^M}(T_t(u,\rho,\mu),T_t(\tilde u, \tilde\rho, \tilde\mu))\leq C_{M,T}d_{\D^M}((u,\rho,\mu),(\tilde u,\tilde\rho,
    \tilde \mu))
  \end{equation} for all $t\in [0,T]$. Let $(u,\rho,\mu)(t)=T_t(u_0,\rho_0,\mu_0)$, then $(u(t,x),\rho(t,x))$ is weak solution
  of the Camassa--Holm equation \eqref{eq:rewchsys10}.
\end{theorem}

We conclude the section about this metric by mentioning that, even if the construction of the metric
is abstract, it can be compared with standard norms, cf. \cite[Section 5]{GHR:13}, so that it can be
used in practice, for example in the study of numerical schemes \cite{CR:12,HR:08}.

\end{document}